\documentclass[11pt,fleqn,leqno]{article} % fleqn left alignment of formulas
\usepackage[final]{graphicx}
\usepackage[cmex10]{amsmath} %option to prevent bitmaps in footnotes
\usepackage{amssymb}
\usepackage{amsfonts}
\usepackage{amsthm}
\usepackage{setspace}
\usepackage{epstopdf}
\usepackage{natbib}
\usepackage[affil-it]{authblk}
\usepackage[final,colorlinks]{hyperref}
\usepackage[margin=1in]{geometry}
\DeclareMathOperator{\Var}{Var}
\DeclareMathOperator{\Cov}{Cov}
\DeclareMathOperator{\corr}{corr}
\DeclareMathOperator{\E}{E}
\DeclareMathOperator{\num}{num}
\newcommand{\abs}[1]{\left\lvert#1\right\rvert}
\newcommand{\norm}[1]{\left\lVert#1\right\rVert}

\newtheoremstyle{nonum}{}{}{\itshape}{}{\bfseries}{.}{ }{\thmnote{#3}}
\theoremstyle{nonum}
\newtheorem{thm}{}
\newtheorem{lemma}{}
\newtheorem{proposition}{}

\begin{document}
\title{Elementary proof of convergence to the mean-field model for the SIR process}
\author[$\dag$]{Ekkehard Beck\footnote{Email addresses: ebeck@u.northwestern.edu (Ekkehard Beck), armbrusterb@gmail.com (Benjamin Armbruster)}}
\author[ ]{Benjamin Armbruster}
\affil[$\dag$]{Department of Industrial Engineering and Management Sciences,
Northwestern University, Evanston, IL, 60208, USA}
\date{\today}

\maketitle

\begin{abstract}
The susceptible-infected-recovered (SIR) model has been used extensively to model disease spread and other processes. Despite the widespread usage of this ordinary differential equation (ODE) based model which represents the mean-field approximation of the underlying stochastic SIR process on contact networks, only few rigorous approaches exist and these use complex semigroup and martingale techniques to prove that the expected fraction of the susceptible and infected nodes of the stochastic SIR process on a complete graph converges as the number of nodes increases to the solution of the mean-field ODE model. Extending the elementary proof of convergence for the SIS process introduced by \citet{Armbruster2015} to the SIR process, we show convergence using only a system of three ODEs, simple probabilistic inequalities, and basic ODE theory. Our approach can also be generalized to many other types of compartmental models (e.g., susceptible-infected-recovered-susceptible (SIRS)) which are linear ODEs with the addition of quadratic terms for the number of new infections similar to the SI term in the SIR model.
\end{abstract}

\section{Introduction} % sec 1
In 1927, \citet{Kermack1927} introduced the susceptible-infected-recovered (SIR) model to study the plague and cholera epidemics in London and Bombay. Since then, the deterministic, ordinary differential equation (ODE) based SIR model has been extensively used to study the spread of infectious diseases having person-to-person transmission mechanisms \citep{AndMay1991} and to study similar processes in fields such as host-pathogen biology \citep{May1983}, chemistry \citep{Cardelli2008}, communications and computer science \citep{Kephart1993}, and social sciences \citep{Daley1964,Daley1965}.

Despite the importance of understanding the accuracy of such mean-field ODE models, most studies that address this question rely on numerical experiments and only few studies exist that rigorously analyze the accuracy of mean-field ODE models compared to the exact Markov models of the disease spread \citep{simon2010}. For the continuous-time stochastic SIR process on a complete network, \cite{Kurtz1970} was the first to show that the fraction of nodes in each disease state converges in probability uniformly over finite time intervals to the solution of the mean-field ODE model as the number of nodes increases. This proof relies on operator semigroup  techniques. In the context of the SIS process, \cite{simon2010} give a less technical summary of this proof and the proofs in \cite{Kurtz1971} and \cite{EthierKurtz1986}.

Related to the above approaches are those by \cite{Bena2008} and \cite{Bortolussi2013} which prove for the same setting convergence in probability \citep{Bortolussi2013} and in mean-square \citep{Bena2008}. Starting out with a discrete-time Markov chain (DTMC), they prove convergence of the corresponding continuous-time Markov chain (CTMC) to the mean-field ODE model using martingale and stochastic approximation algorithms techniques. Similar, convergence results also exist for variants of the SIR-models different to the classical one such as the SIR model on a configuration model network  specified by an arbitrary degree distribution \citep{Volz2008} where convergence to the deterministic limit was proved by \cite{Decreusefond2012} using again complex techniques drawn from stochastic differential equations (SDE) and martingale theory.  We thank a reviewer for mentioning the proof by \cite{Andersson2000} using time-changed Poisson processes, which is reasonably elementary.

Recently, \cite{Armbruster2015} introduced an elementary approach to show convergence for the SIS process.  Using only a system of two ODEs, basic ODE techniques, and Jensen's inequality, they show that the dynamics of the fraction of infected in the continuous-time stochastic SIS-process on a complete graph converges uniformly over finite time intervals in mean-square to the mean-field ODE as the population size increases. However, that approach cannot be easily extended since it requires the mean-field ODE to be one-dimensional.  This inspired us to develop in this paper a more general approach for the SIR process and prove that the fractions of susceptibles and infected nodes of the stochastic SIR-process on a complete graph converges uniformly over finite time intervals in mean-square to the solution of the mean-field ODE model as the number of nodes in the network increases. The approach remains elementary, using only a system of three ODEs describing the dynamics of the fraction of susceptibles and infected and the dynamics of the total stochasticity in the system; simple probability inequalities to bound the difference of the means of products such as $\E[SI]$ from the corresponding product of means, $\E[S]\E[I]$; and basic ODE techniques.

Extending our previous approach \citep{Armbruster2015}) to the SIR process forwards our agenda of showing that mean-field convergence results can be tackled using only basic probability and ODE theory. We hope to open up the field to more researchers that might not be comfortable using semigroup or martingale theory.

\section{Main result}\label{sec:mainResult}
Consider a complete graph of $n$ nodes and let $X_n(t)$ be a Markov process describing the state of nodes at time $t$. Each node can be in one of the three states: susceptible $S$, infected $I$, and recovered $R$ (immune). An infected node infects each susceptible neighbor at a rate $\tau/n$ and recovers at rate $\gamma$.  The $1/n$ scaling of the per edge transmission rate is necessary to obtain a mean-field limit.
We define $S_n(t):=\num_S (X_n(t))$ and $I_n(t):=\num_I (X_n(t))$ as the number of susceptible and infected nodes, and $s_n(t):=S_n(t)/n$ and $i_n(t):=I_n(t)/n$ as the susceptible and the infected fraction, respectively.  Clearly, the number of recovered nodes is $R_n(t):=\num_R (X_n(t))=n-S_n(t)-I_n(t)$. Over the course of this paper we may sometimes drop the dependence on $t$ and $n$ in our notation.  The exact trajectories of $\E[S_n(t)]$ and $\E[I_n(t)]$ can be calculated using the Kolmogorov (or master) equations, a system of $n^2$ differential equations:
\begin{multline}\label{eq:master}
	P[S_n = i,\ I_n=j]' = -((\tau/n)ij+\gamma j)P[S_n = i,\ I_n=j]\\
 + (\tau/n)(i+1)(j-1)P[S_n = i+1,\ I_n=j-1]\\
  + \gamma (j+1) P[S_n = i,\ I_n=j+1]
 \quad \text{for all $0\leq i,j\leq n$.}
\end{multline}

Our goal is to prove the following theorem, which shows that the dynamics of the expectation of the fraction of susceptible and infected nodes of the stochastic SIR process converges in mean-square on a complete graph and finite-time intervals to the mean-field approximation as the population size increases.

\begin{thm}[Theorem 1]\label{thm1}
If $s_n(0)\to s_0$ and $i_n(0)\to i_0$ as $n\rightarrow\infty$ and $s_0$ and $i_0 \in [0,1]$, then 
$(s_n(t),i_n(t))$ converges uniformly in mean square, i.e., 
\begin{equation*}
\E[\norm{(s_n(t),i_n(t))-y(t)}^2]\rightarrow 0,
\end{equation*}
on any time finite interval $[0,T]$ to the solution of the mean-field
equations: 
\begin{subequations}\label{eq:1}
\begin{align} 
\begin{split}\label{eq:1a}
y_1'=-\tau y_1y_2,\quad y_1(0)=s_0, 
\end{split}\\
\begin{split} \label{eq:1b}
y_2'= \tau y_1y_2-\gamma y_2,\quad y_2(0)=i_0. 
\end{split}
\end{align}
\end{subequations}
\end{thm}

\section{Proof of the main result} %sec 3
\subsection{Preliminaries}
In this subsection we outline our approach and handle some routine technicalities.  In the following subsection we establish the set of differential equations which describe the dynamics of the first and second moments of the fractions of nodes of the stochastic SIR process.  The problem with those differential equations is the presence of $\E[si]$, $\E[s^2i]$, and $\E[si^2]$ terms on the right hand sides.
Then in Section~\ref{sec:bounds} we develop bounds for these terms.  These then allow us in Section~\ref{sec:closure} to obtain a system of ODEs equivalent to that in Section~\ref{sec:diffeqs} but which is closed, so that the right hand sides are functions of the variables on the left hand sides.
The system of ODEs describes the dynamics of the expectations and the dynamics of the sum of the variances.
Finally in Section~\ref{sec:convergence}, we bound the  solution of these ODEs to the solution of the mean-field equations \eqref{eq:1}.

\paragraph{Existence and uniqueness.}
Since the right hand sides of equations (\ref{eq:1}) are smooth (i.e., continuously differentiable), a unique solution exists to the mean-field equations. Further, since it is easy to see that $(y_1(t),y_2(t))$ remains in $[0,1]^2$, the solution exists for all times $t$.

\subsection{Initial set of differential equations}\label{sec:diffeqs}
Let $\num_{SI}(X(t))$ be the number of edges which connect a susceptible and an infected individual. Since we are on a complete network, $\num_{SI}(X(t))=\num_S(X(t))\num_I(X(t))$, or $S\cdot I$ in our abbreviated notation.  We now start by describing aspects of the stochastic process, \eqref{eq:master}, from Section~\ref{sec:mainResult} using the following differential equations of the means.

\begin{proposition}[Proposition 2]\label{prop2}
The following differential equations hold:
\begin{subequations}\label{eq:2}
\begin{align}
\label{eq:2a}
\E[S]'&=-(\tau/n)\E[SI],\\
\label{eq:2b}
\E[I]'&=(\tau/n)\E[SI]-\gamma \E[I],\\
\label{eq:2c}
\E[S^2]'&=-(\tau/n)(2\E[S^2I]-\E[SI]),\\
\label{eq:2d}
\E[I^2]'&=-(\tau/n)(2\E[SI^2]+\E[SI])-\gamma(2\E[I^2]-\E[I]).
\end{align}
\end{subequations}
\end{proposition}

\begin{proof}
These equations can be found in papers focusing on pair-models which build on them \citep{Keeling1999b,Rand1999}.  Our proof is more formal and takes the same approach as in Proposition 2 of \cite{Armbruster2015}, which proves a similar result for the SIS process. Formally, for each equation in (\ref{eq:2}) we are merely using the Kolmogorov backward equations (see any undergraduate textbook such as \cite{Ross2007}) where for a Markov process $X(t)$ with transition rate matrix $Q$ and any function $a(x)$, 
\begin{equation*}
 \E[a(X(t))]'=\E[b(X(t))],
\end{equation*}
where we define
\begin{equation*}
b(x):=\sum_{\hat{x}}Q(x,\hat{x}) a(\hat{x}).
\end{equation*}
Using the fact that $Q(x,x)=-\sum_{\hat{x}\neq x}Q(x,\hat{x})$,
\begin{equation*}
b(x)=\sum_{\hat{x}\neq x}Q(x,\hat{x})(a(\hat{x})-a(x)).
\end{equation*}

The terms in the sum are products of the transition rates and the size of the resulting changes in $a(X(t))$. In our case, we only need to handle two types of transitions: the network configuration, $x$, can transition to states $\hat{x}$ with either an additional infection or an infection that recovered. The first occurs at an aggregate rate of $(\tau/n) \num_{SI}(x)$ while the second at rate $\gamma \num_I(x)$. Transitions due to an additional infection increase the number of infected and decrease the number of susceptibles by one, $\num_I(\hat{x})-\num_I(x)=1$ and $\num_S(\hat{x})-\num_S(x)=-1$. Choosing $a(x)=\num_I(x)$ and $a(x)=\num_S(x)$, then explains (\ref{eq:2a}) and (\ref{eq:2b}), respectively. We explain (\ref{eq:2c}) and (\ref{eq:2d}) by choosing $a(x)=\num_S(x)^2$ and $a(x)=\num_I(x)^2$, respectively, and noting that for an additional infection,
\begin{equation*}
\num_S(\hat{x})^2-\num_S(x)^2=(\num_S(x)-1)^2-\num_S(x)^2=-(2\num_S(x)-1),
\end{equation*}
\begin{equation*}
\num_I(\hat{x})^2-\num_I(x)^2=(\num_I(x)+1)^2-\num_I(x)^2=2\num_I(x)+1,
\end{equation*}
while for a transition due to an infection that recovered, $\num_S(\hat{x})=\num_S(x)$ and
\begin{equation*}
\num_I(\hat{x})^2-\num_I(x)^2=(\num_I(x)-1)^2-\num_I(x)^2=-(2\num_I(x)-1).
\end{equation*}
\end{proof}

We normalize \eqref{eq:2a}--\eqref{eq:2b} and \eqref{eq:2c}--\eqref{eq:2d} by dividing by $n$ and $n^2$ respectively:
\begin{subequations}\label{eq:3}
\begin{align}
\begin{split}\label{eq:3a}
\E[s]'=-\tau \E[si],
\end{split}\\
\begin{split}\label{eq:3b}
\E[i]'=\tau \E[si]-\gamma \E[i],
\end{split}\\
\begin{split}\label{eq:3c}
\E[s^2]'=-\tau(2\E[s^2i]-\E[si]/n),
\end{split}\\
\begin{split}\label{eq:3d}
\E[i^2]'=\tau(2\E[si^2]+\E[si]/n)-\gamma(2\E[i^2]-\E[i]/n).
\end{split}
\end{align}
\end{subequations}
This highlights the $O(1/n)$ terms that disappear as $n\rightarrow \infty$. Figure~\ref{fig1} illustrates this for a numerical example.  The difficulty with these equations, \eqref{eq:3}, is that they are not closed, because they do not describe how to evaluate the product terms (i.e., $\E[si]$, $\E[s^2i]$, and $\E[si^2]$) on the right hand sides.

\begin{figure}[h]
	\centering
	\includegraphics[scale=0.5]{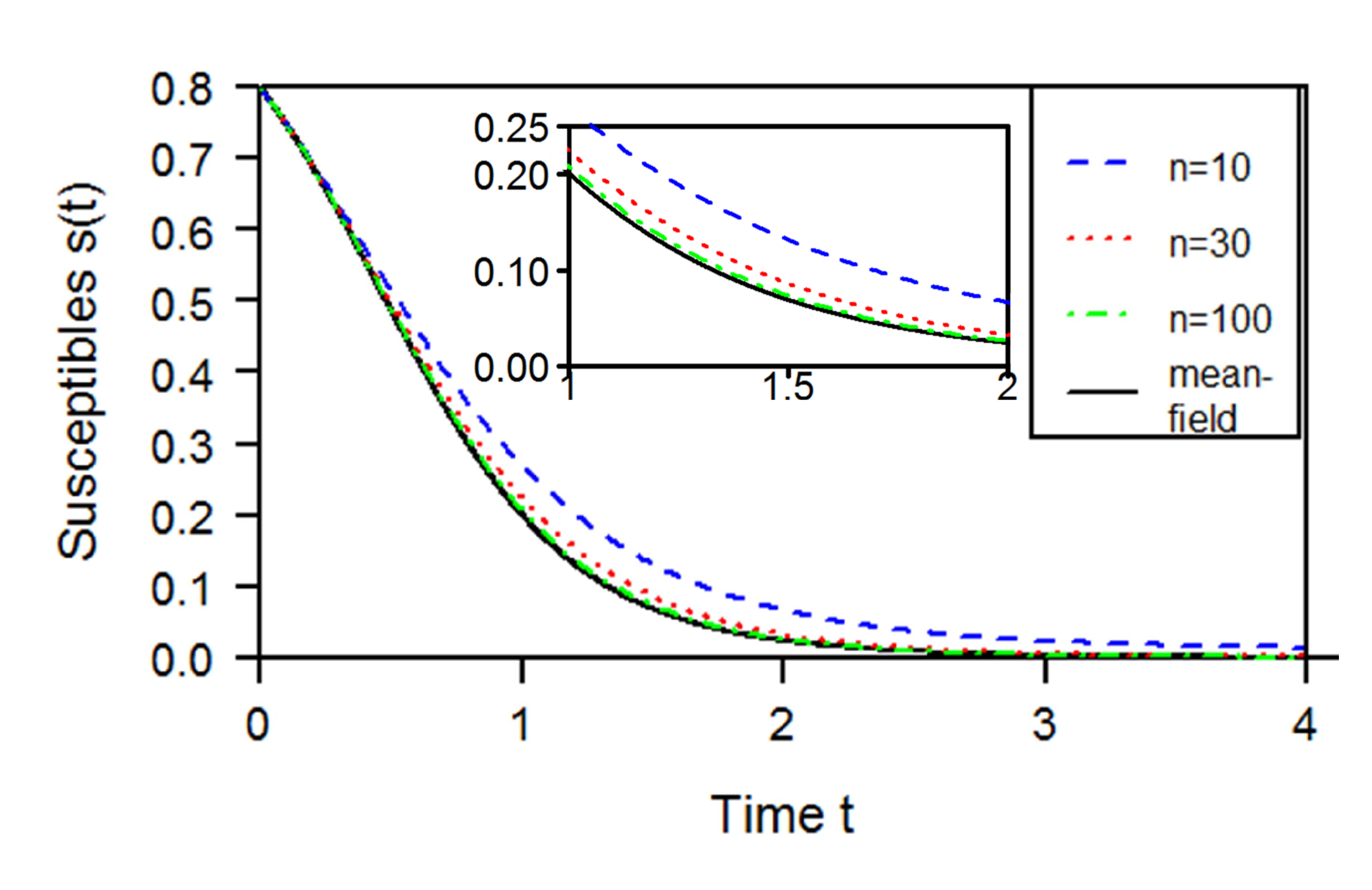}
	\caption{SIR epidemic fraction of susceptibles with $\tau = 3, \gamma=0.25$. The solution of the mean-field equations (\ref{eq:1}) (black curve) in comparison to the expected fraction, $\E[s_n(t)]$, calculated using the Kolmogorov equations \eqref{eq:master}, for increasing number of nodes: $n=10$ (blue curve), $n=30$ (red curve), and $n=100$ (green curve).}
	\label{fig1}
\end{figure}

\subsection{Bounds on the expectation of products}\label{sec:bounds}

Using \nameref{lemma5} we bound the difference of the expectations of the product terms in (\ref{eq:3}) (i.e., $\E[si]$, $\E[s^2i]$, and $\E[si^2]$) to the corresponding product of expectations (i.e., $\E[s]\E[i]$, $\E[s]^2\E[i]$, and $\E[s]\E[i]^2$).

\begin{lemma}[Lemma 3]\label{lemma3}
For random variables $Y$ and $Z$, 
$\Var[Y+Z] \le 2(\Var[Y]+\Var[Z])$.
\end{lemma}
\begin{proof}
 Note that,
\begin{equation*}
\begin{split}
\Var[Y+Z] & =\Var[Y]+\Var[Z]+2\corr(Y,Z)\sqrt{\Var[Y]\Var[Z]},\\
&\le \Var[Y]+\Var[Z]+2\sqrt{\Var[Y]\Var[Z]},\\
&\le \Var[Y]+\Var[Z]+2(\Var[Y]+\Var[Z])/2,\\
&=2(\Var[Y]+\Var[Z]),
\end{split}
\end{equation*}
where the last inequality holds because the geometric mean is less than or equal to the arithmetic mean.
\end{proof}

\begin{lemma}[Lemma 4]\label{lemma4}
Consider a random variable $Y$ in $[0,1]$. Then 
$\Var[Y^2]\le 4\Var[Y]$.
\end{lemma}
\begin{proof}
This proof is adapted from \cite{Giraudo2014}. Note that $Y^2=(Y-\E[Y])Y+\E[Y]Y$. Applying \nameref{lemma3}, we obtain
\begin{equation}\label{eq:7}
\begin{aligned}
\Var[Y^2]\le 2(\Var[(Y-\E[Y])Y]+\Var[\E[Y]Y]).
\end{aligned}
\end{equation}
From the definition,
\begin{align*}
\Var[(Y-\E[Y])Y]&=\E[((Y-\E[Y])Y)^2]-\E[(Y-\E[Y])Y]^2\\
&\le \E[((Y-\E[Y])Y)^2]\le \E[(Y-\E[Y])^2]=\Var[Y]
\end{align*}
where the last inequality is due to $\abs{Y}\le 1$. Substituting back into (\ref{eq:7}), we obtain
\begin{equation*}
\Var[Y^2]\le 2(\Var[Y]+\E[Y]^2 \Var[Y]) \le 4 \Var[Y].
\end{equation*} 
\end{proof}

The following example shows that the factor of 4 in \nameref{lemma4} is tight. Consider $P[Y=1]=P[Y=1 - 2\delta]=1/2$. Then $\Var[Y]=\delta^2$ and $\Var[Y^2]=4\delta^2(1-2\delta+\delta^2)$, which asymptotically equals $4\delta^2$ as $\delta \rightarrow 0$.

\begin{lemma}[Lemma 5]\label{lemma5}
For random variables $Y$ and $Z$ in $[0,1]$,
\begin{subequations}\label{eq:4}
\begin{align}
\label{eq:4a}
\abs{\E[YZ]-\E[Y]\E[Z]}&\le (\Var[Y]+\Var[Z])/2,\\
\label{eq:4b}
\abs{\E[Y^2Z]-\E[Y]^2\E[Z]}&\le 2(\Var[Y]+\Var[Z]).
\end{align}
\end{subequations}
\end{lemma}
\begin{proof}  
From the facts that $\Cov[Y,Z]=\E[YZ]-\E[Y]\E[Z]$, $\Cov[Y,Z]=\corr[Y,Z]\sqrt{\Var[Y]\Var[Z]}$, and $\abs{\corr[Y,Z]}\le 1$, we have 
\begin{equation}\label{eq:5}
	\abs{\E[YZ]-\E[Y]\E[Z]}\le \sqrt{\Var[Y]\Var[Z]}.
\end{equation}
Applying \nameref{lemma4}, 
\begin{equation}\label{eq:6}
	\abs{\E[Y^2Z]-\E[Y^2]\E[Z]}\le 2\sqrt{\Var[Y]\Var[Z]}.
\end{equation}
Using the triangle inequality and the fact that $\abs{\E[Z]}\leq 1$,
\begin{multline*}
	\abs{\E[Y^2Z]-\E[Y]^2\E[Z]}\le 
	\abs{\E[Y^2Z]-\E[Y^2]\E[Z]}+\abs{\E[Y^2]\E[Z]-\E[Y]^2\E[Z]}\\
	\leq 2\sqrt{\Var[Y]\Var[Z]}+\Var[Y]\abs{\E[Z]}
	\leq 2\sqrt{\Var[Y]\Var[Z]}+\Var[Y]+\Var[Z].
\end{multline*}
The fact that the arithmetic mean is greater than the geometric mean,
\begin{equation*}
\sqrt{\Var[Y]\Var[Z]} \le (\Var[Y]+\Var[Z])/2,
\end{equation*}
then proves the claims.
\end{proof}

We now apply \nameref{lemma5} to $\E[s_n(t)i_n(t)]$, $\E[s_n(t)^2i_n(t)]$, and $\E[s_n(t)i_n(t)^2]$.  We can replace the inequality in \eqref{eq:4} by defining the corresponding functions $h_{1,n}(t), h_{2,n}(t),$ and $h_{3,n}(t)$ in $[-1,1]$ such that
\begin{subequations}\label{eq:8}
\begin{align}
\label{eq:8a}
\E[s_n(t)i_n(t)]&=\E[s_n(t)]\E[i_n(t)]+h_{1,n}(t)(\Var[s_n(t)]+\Var[i_n(t)])/2,\\
\label{eq:8b}
\E[s_n(t)^2i_n(t)]&=\E[s_n(t)]^2\E[i_n(t)]+2h_{2,n}(t)(\Var[s_n(t)]+\Var[i_n(t)]),\\
\label{eq:8c}
\E[s_n(t)i_n(t)^2]&=\E[s_n(t)]\E[i_n(t)]^2+2h_{3,n}(t)(\Var[s_n(t)]+\Var[i_n(t)]).\\
\intertext{We also define $h_{4,n}(t)$ in $[0,1]$ such that} 
\label{eq:8d}
\Var[i_n(t)]&=h_{4,n}(t)(\Var[s_n(t)]+\Var[i_n(t)]),
\end{align}
\end{subequations}
which we will require in the next step, where we establish a system of ODEs equivalent to equations (\ref{eq:3}) using only the two mean-field equations (\ref{eq:3a})--(\ref{eq:3b}) and an equation describing the dynamics of the total variance in the system.

\subsection{Closed system of ODEs}\label{sec:closure}

Since the inequalities in \nameref{lemma5} involve the variance, it will be convenient to replace (\ref{eq:3c})--(\ref{eq:3d}) by an equation for $(\Var[s]+\Var[i])'$, representing the dynamics of the total stochasticity in the system. Since $\Var[Y]=\E[Y^2]-\E[Y]^2$, it follows that $\Var[Y]'=\E[Y^2]'-2\E[Y]\E[Y]'$ and thus,
\begin{align}
\nonumber
\Var[s]' &=-\tau(2\E[s^2i]-\E[si]/n)+2\tau \E[s]\E[si] \\
& =-2\tau \E[s^2i]+\tau \E[si](2\E[s]+1/n), 
\nonumber\\
\nonumber
\Var[i]' &=\tau (2 \E[si^2]+\E[si]/n)-\gamma(2 \E[i^2]- \E[i]/n)-2\E[i](\tau \E[si]-\gamma \E[i]) \\
& = 2\tau \E[si^2]+\tau \E[si](-2 \E[i]+1/n)-2\gamma \Var[i]+\gamma \E[i]/n, 
\nonumber\\
\label{eq:9c}
(\Var[s]+\Var[i])'& =2\tau (\E[si^2]- \E[s^2i]) \\
\nonumber
&\quad +2\tau \E[si](\E[s]- \E[i]+1/n)-2 \gamma \Var[i]+\gamma \E[i]/n.
\end{align}
Hence our current system of differential equations is (\ref{eq:3a}),(\ref{eq:3b}), and (\ref{eq:9c}). We now use the substitutions defined in \eqref{eq:8} for $\E[si]$, $\E[s^2i]$, and $\E[si^2]$ to obtain the following system of ODEs:
\begin{subequations}\label{eq:10}
\begin{align}
\label{eq:10a}
\E[s]'&=-\tau(\E[s]\E[i]+h_{1}(\Var[s]+\Var[i])/2),\\
\label{eq:10b}
\E[i]'&=\tau(\E[s]\E[i]+h_{1}(\Var[s]+\Var[i])/2)-\gamma \E[i],\\
\nonumber
(\Var[s]+\Var[i])'&=2\tau(\E[s]\E[i]^2+2h_{3}(\Var[s]+\Var[i]))\\
\nonumber
&\quad -2\tau(\E[s]^2\E[i]+2h_{2}(\Var[s]+\Var[i]))\\
\nonumber
&\quad +2\tau(\E[s]\E[i]+h_{1}(\Var[s]+\Var[i])/2)(\E[s]-\E[i]+1/n)\\
\nonumber 
& \quad -2\gamma h_{4}(\Var[s]+\Var[i])+\gamma \E[i]/n, \\
\nonumber
\quad &=(4\tau(h_{3}-h_{2})+\tau h_{1}(\E[s]-\E[i]+1/n)-2\gamma h_{4})\\
\label{eq:10c}
&\quad \cdot(\Var[s]+\Var[i])+(2\tau \E[s]\E[i]+\gamma \E[i])/n.
\end{align}
\end{subequations}

Equations (\ref{eq:10a})--(\ref{eq:10c}) define a proper system of differential equations with the same variables ($\E[s]$, $\E[i]$, and $\Var[s]+\Var[i]$) on the right hand side as on the left hand side. To be more explicit, we define the vector $z_n :=(\E[s_n],\E[i_n],\Var[s_n]+\Var[i_n])$. As with the original state variables $S$ and $i$, we may sometimes write $z$ or $h_k$, dropping the dependence on $n$ and $t$. Thus, $z_n(t)$ is a solution of the initial value problem, $z'=g_n(t,z;1/n)$ and $z(0)=z_{0,n}$, where 
\begin{subequations}\label{eq:11}
\begin{align}
\label{eq:11a}
g_{1,n}(t,z;\epsilon)&:=-\tau(z_1z_2+h_{1,n}(t)z_3/2),\\
\label{eq:11b}
g_{2,n}(t,z;\epsilon)&:=\tau(z_1z_2+h_{1,n}(t)z_3/2)-\gamma z_2,\\
\label{eq:11c}
g_{3,n}(t,z;\epsilon)&:=(4\tau(h_{3,n}(t)-h_{2,n}(t))+\tau h_{1,n}(t)(z_1-z_2+\epsilon)\\
\nonumber
&\quad -2\gamma h_{4,n}(t))z_3+(2\tau z_1z_2+\gamma z_2)\epsilon,\\
\label{eq:11d}
z_{0,n}&:=(s_n(0),i_n(0),0).
\end{align}
\end{subequations} 

Note that the mean-field solution, $\bar{z}(t):=(y_1(t),y_2(t),0)$, solves (for any $n$) the initial  value problem $z'=g_n(t,z;0)$ and $z(0)=z_0$, where $z_0:=(s_0,i_0,0)$. Our next step is to bound $\norm{z_n-\bar{z}}$.

\subsection{Convergence}\label{sec:convergence}
\begin{lemma}[Lemma 6]\label{lemma6}
Consider the initial value problems $x'=f_1(t,x)$, $x(0)=x_1$ and $x'=f_2(t,x)$, $x(0)=x_2$ with solutions $\varphi_1(t)$ and $\varphi_2(t)$ respectively.  If $f_1$ is Lipschitz in $x$ with constant $L$ and \linebreak $\norm{f_1(t,x)-f_2(t,x)}\le M$, then $\norm{\varphi_1(t)-\varphi_2(t)}\le (\norm{x_1-x_2}+M/L)e^{Lt}-M/L$.
\end{lemma}
\begin{proof} We give a proof in the appendix using an argument often found in the proof of the standard ODE existence theorem.
\end{proof}

To apply \nameref{lemma6} we first note that the domain of $z$ for $g_n(t,z;\epsilon)$ is bounded: $\E[s]$ and $\E[i]$ are in $[0,1]$ and $\Var[s]+\Var[i]$ is in $[0,2]$. The domain for $\epsilon$ is $[0,1]$. Since $g_n(t,z;0)$ is a polynomial in $z$, we can determine a Lipschitz constant with respect to $z$ by bounding $\sum_{j=1}^3 \sum_{k=1}^3 \abs{\frac{\delta g_{j,n}}{\delta x_k}}$. Using the bounds on the domain, on $\epsilon$, and on the functions $h_{k,n}$, $k=1,\ldots,4$, we determine that $L=22\tau + 2\gamma$ is such a Lipschitz constant. Using the bounds on $z$ and the functions $h_{k,n}$, we define $M(\epsilon)$ as follows:
\begin{equation*}
\begin{split}
\norm{g_n(t,z;\epsilon)-g_n(t,z;0)}&=\abs{g_{3,n}(t,z;\epsilon)-g_{3,n}(t,z;0)} \\
&=\abs{\tau h_{1,n}(t)\epsilon z_3 + (2\tau z_1 z_2+\gamma z_2)\epsilon} \\
&\le (4\tau+\gamma)\epsilon =: M(\epsilon). \nonumber
\end{split}
\end{equation*}

We now apply \nameref{lemma6} by letting $f_1(t,x)=g_n(t,x;0)$, $f_2(t,x)=g_n(t,x,1/n)$, $x_1=z_0$, $x_2=z_{0,n}$, $\varphi_1 = \bar{z}$, and $\varphi_2 = z_n$. Then the difference between $z_n(t)$ and $\bar{z}(t)$ is
\begin{equation*}
\norm{z_n(t)-\bar{z}(t)}\le \Bigl(\norm{z_0-z_{0,n}}+\frac{M(1/n)}{L}\Bigr)e^{Lt}-\frac{M(1/n)}{L},
\end{equation*}

Thus, for $t\le T$,
\begin{equation}\label{eq:12}
\norm{z_n(t)-\bar{z}(t)}\le \Bigl(\norm{z_0-z_{0,n}}+\frac{M(1/n)}{L}\Bigr)e^{LT}-\frac{M(1/n)}{L}.
\end{equation}
Since $M(\epsilon)\to 0$ as $\epsilon \rightarrow 0$ and $z_{0,n}\to z_0$, the right hand side of (\ref{eq:12}) goes to zero as $n \rightarrow \infty$. Hence we have uniform convergence of $z_n \rightarrow \bar{z}$ over any finite time interval $[0,T]$.

We now show that $z_n \rightarrow \bar{z}$ implies convergence in mean-square, i.e., $\E[\norm{(s_n, i_n)}^2] \rightarrow 0$. In finite dimensions, all norms are equivalent for the purposes of proving convergence. For convenience we choose the 2-norm for $\norm{(s_n, i_n) - y}_2$ and the 1-norm for $\norm{z_n-\bar{z}}_1$. Recalling that $\Var[Y]=\E[Y^2]-\E[Y]^2$, we decompose the mean-square error into a bias and a variance term:
\begin{equation*}
\begin{split}
\E[\norm{(s_n,i_n)-y}_2^2]&=\E[(s_n-y_1)^2]+\E[(i_n-y_2)^2] \\
&=(\E[s_n]-y_1)^2+(\E[i_n]-y_2)^2+\Var[s_n]+\Var[i_n] \\
&\leq\abs{\E[s_n]-y_1}+\abs{\E[i_n]-y_2}+\Var[s_n]+\Var[i_n] \\
&=\abs{z_{n,1}-\bar{z}_1}+\abs{z_{n,2}-\bar{z}_2}+\abs{z_{n,3}-\bar{z}_3}= \norm{z_n-\bar{z}}_1 
\end{split}
\end{equation*}
where of course $\bar{z}_3=0$, and the inequality holds because $\E[s_n]$, $y_1$, $\E[i_n]$, and $y_2$ are in $[0,1]$. Thus uniform convergence of $\norm{z_n-\bar{z}}_1 \rightarrow 0$ for $t\le T$ implies the same for $\E[\norm{(s_n,i_n)-y}_2^2] \rightarrow 0$, proving our claim.

\section{Conclusion}

We extended the elementary approach of \cite{Armbruster2015} to show that the expected fractions of nodes of the stochastic SIR process on a complete graph converge uniformly in mean-square on finite time intervals to the solution of the mean-field ODE model. We set up differential equations for the first and second moments of the infected and susceptibles.  Our main tool was \nameref{lemma5}, two simple probabilistic inequalities involving expectations and variances. They let us to bound in the right hand side of the differential equations, the difference between the stochastic terms $\E[si]$, $\E[si^2]$, and $\E[si^2]$ and the corresponding mean-field terms.

\begin{figure}
	\includegraphics[width=\textwidth]{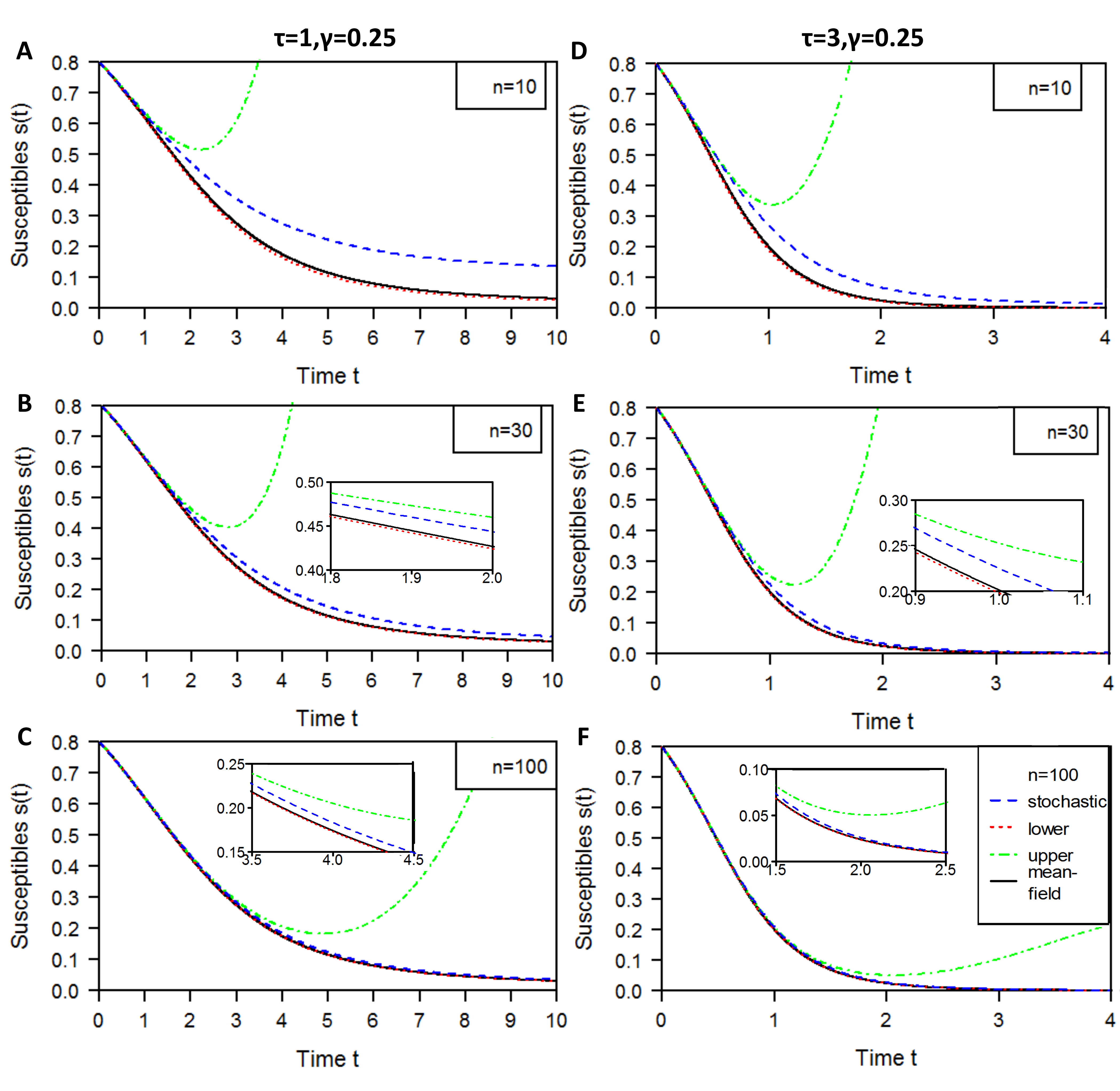}
	\centering
	\caption{Fraction of susceptibles of SIR process: the true expectation, $\E[s_n(t)]$ (blue dashed curve); the mean-field approximation, $y_1(t)$ (black curve); apparent upper (green dashed-dotted curve) and lower (red dotted curve) bounds on $\E[s_n(t)]$.  The upper and lower bounds correspond to component $z_1(t)$ of solutions to $z'=g_n(t,z;1/n),z(0)=z_0$ (see \eqref{eq:11}) where we set $(h_{1,n},h_{2,n},h_{3,n},h_{4,n})$ to constants: $(0.5,1,0.5,1)$ for the lower bound and $(-1,-0.8,-0.4,0.5)$ for the upper bound in panels A, B, and C; and $(0.5,1,0.5,1)$ for the lower bound and $(-1,-1,-0.8,1)$ for the upper bound in panels D, E, and F.  The panels vary the number of nodes ($n=10,30,100$) and the transmission rate ($\tau=1,3$).  The recovery rate is $\gamma=0.25$.}
	\label{fig2}
\end{figure}

We want to note that our elementary approach generalizes to related epidemic models such as the susceptible-infected-recovered-susceptible (SIRS), susceptible-exposed-infected-susceptible (SEIS), and susceptible-exposed-infected-recovered (SEIR) model.  The corresponding mean-field models are linear ODEs with the addition of quadratic terms for the number of new infections (i.e., the $S\cdot I$ term in equations (\ref{eq:1a})--(\ref{eq:1b})). We go through the key steps to generalization using the SEIR model as an example.  First, the existing differential equations in the ODE systems (\ref{eq:2}) and (\ref{eq:3}) have to be modified for the SEIR model and differential equations have to be added which describe the dynamics of the first and second moments of the additional state variable (i.e., the fraction of nodes exposed). Applying \nameref{lemma5} we again bound the difference of the stochastic $\E[si]$, $\E[s^2i]$, and $\E[si^2]$ terms from the mean-field equations since the SEIR and all the other models discussed above have the same or a similar quadratic terms for the rate of new infections. To establish the ODE system equivalent to (\ref{eq:10}), we first add the variance of the fraction exposed to the equation for the total stochasticity (\ref{eq:9c}).  We then add the variance of the exposed fraction to the right hand side of (\ref{eq:8d}) and add a similar equation defining $h_{5,n}$ with the variance of the exposed fraction on the left hand side.  Thus, the resulting equivalent ODE system consists of the mean-field equations describing the dynamics of all the expected fraction of nodes as in style of (\ref{eq:10a}) and (\ref{eq:10b}) and one equation describing the dynamics of the total variance in the system (\ref{eq:10c}). Then Section~\ref{sec:convergence} can be used to prove convergence.

Going further, we point out that some choices for $h_{k,n}(t)$ when solving \eqref{eq:11} appear to give lower or upper bounds on $\E[s_n(t)]$ or $\E[i_n(t)]$ (see \nameref{lemma5}, (\ref{eq:8a})--(\ref{eq:8c}), and (\ref{eq:10})). Figure \ref{fig2} shows the fraction of susceptibles for some numerical examples where this appears to be the case.  In those examples we set the $h_{k,n}(t)$ to constants determined by trial-and-error. It is worth noting that while the upper bounds are only tight until some finite time $T$ (which of course increases as $n\to\infty$) after which they diverge to infinity, the lower bounds appear to be tight for all times. Whether these observations can be proven is an area for further research.  An additional area for future research is proving that the uniform convergence holds for all times and not just a finite interval.

This extension of our previously introduced elementary approach \citep{Armbruster2015} to the SIR case forwards our agenda of showing that mean-field convergence results can be tackled using only basic ODE theory. Using this approach, we hope to open up the field to researchers more comfortable with ODEs than stochastic processes and allow more researchers to rigorously analyze the accuracy of compartmental models.

\textbf{Acknowledgements.} We thank Peter L. Simon, Tom Britton, and two anonymous referees for helpful comments.

\bibliography{refSIR}

\begin{thebibliography}{}

\bibitem[\protect\astroncite{Anderson and May}{1991}]{AndMay1991}
Anderson, R.~M. and May, R.~M. (1991).
\newblock {\em Infectious Diseases of Humans {D}ynamics and {C}ontrol}.
\newblock Oxford University Press.

\bibitem[\protect\astroncite{Andersson and Britton}{2000}]{Andersson2000}
Andersson, H. and Britton, T. (2000).
\newblock {\em Stochastic epidemic models and their statistical analysis.
  Lecture Notes in Statistics}, chapter~5.
\newblock Springer.

\bibitem[\protect\astroncite{Armbruster and Beck}{2016}]{Armbruster2015}
Armbruster, B. and Beck, E. (2016).
\newblock An elementary proof of convergence to the mean-field equations for an
  epidemic model.
\newblock {\em The IMA Journal of Applied Mathematics}.
\newblock In press.

\bibitem[\protect\astroncite{Bena\"{i}m and Le~Boudec}{2008}]{Bena2008}
Bena\"{i}m, M. and Le~Boudec, J.-Y. (2008).
\newblock A class of mean-field interaction models for computer and
  communication systems.
\newblock {\em Performance Evaluation}, 65(11-12):823--838.

\bibitem[\protect\astroncite{Bortolussi et~al.}{2013}]{Bortolussi2013}
Bortolussi, L., Hillston, J., Latella, D., and Massink, M. (2013).
\newblock Continuous approximation of collective system behavior: A tutorial.
\newblock {\em Performance Evaluation}, 70(5):317--349.

\bibitem[\protect\astroncite{Cardelli}{2008}]{Cardelli2008}
Cardelli, L. (2008).
\newblock From processes to {ODE}s by chemistry.
\newblock {\em TCNature, International Federation for Information Processing
  (Springer, Boston)}, 273:261--281.

\bibitem[\protect\astroncite{Daley and Kendall}{1964}]{Daley1964}
Daley, D. and Kendall, D. (1964).
\newblock Epidemics and rumours.
\newblock {\em Nature}, 204(225):1118.

\bibitem[\protect\astroncite{Daley and Kendall}{1965}]{Daley1965}
Daley, D. and Kendall, D. (1965).
\newblock Stochastic rumors.
\newblock {\em IMA}, 1(1):42--55.

\bibitem[\protect\astroncite{Decreusefond et~al.}{2012}]{Decreusefond2012}
Decreusefond, L., Dhersin, J.-S., Moyal, P., and Chi~Tran, V. (2012).
\newblock Large graph limit for an {SIR} process in random network with
  heterogeneous connectivity.
\newblock {\em Annals of Applied Probability}, 22(2):541--575.

\bibitem[\protect\astroncite{Ethier and Kurtz}{1986}]{EthierKurtz1986}
Ethier, S.~N. and Kurtz, T.~G. (1986).
\newblock {\em Markov Processes: Characterization and Convergence}, chapter
  11.2.
\newblock Wiley series in probability and statistics. Wiley.

\bibitem[\protect\astroncite{Giraudo}{2014}]{Giraudo2014}
Giraudo, D. (2014).
\newblock Bound the variance of the product of two random varables.
\newblock Mathematics Stack Exchange.
\newblock URL:http://math.stackexchange.com/q/1044864 (version: 2014-11-30).

\bibitem[\protect\astroncite{Hale}{2009}]{hale2009ordinary}
Hale, J. (2009).
\newblock {\em Ordinary Differential Equations}, chapter 1.6.
\newblock Dover Books on Mathematics Series. Dover Publications.

\bibitem[\protect\astroncite{Keeling}{1999}]{Keeling1999b}
Keeling, M.~J. (1999).
\newblock The effects of local spatial structure on epidemiological invasions.
\newblock {\em Proc. R. Soc. Lond. B}, 266:859--867.

\bibitem[\protect\astroncite{Kephart and White}{1993}]{Kephart1993}
Kephart, J. and White, S. (1993).
\newblock Measuring and modeling computer virus prevalence.
\newblock {\em Proceedings, 1993 IEEE Computer Society Symposium on Research in
  Security and Privacy}, pages 2--15.

\bibitem[\protect\astroncite{Kermack and McKendrick}{1927}]{Kermack1927}
Kermack, W. and McKendrick, A. (1927).
\newblock A contribution to the mathmatical theory of epidemics.
\newblock {\em Proceedings of the Royal Society of London. Series A: Containing
  Papers of a Mathematical and Physical Character}, 115(772):700--721.

\bibitem[\protect\astroncite{Kurtz}{1970}]{Kurtz1970}
Kurtz, T.~G. (1970).
\newblock Solutions of ordinary differential equations as limits of pure jump
  {M}arkov processes.
\newblock {\em Journal of Applied Probability}, 7:49--58.

\bibitem[\protect\astroncite{Kurtz}{1971}]{Kurtz1971}
Kurtz, T.~G. (1971).
\newblock Limit theorems for sequences of jump {M}arkov processes approximating
  ordinary differential processes.
\newblock {\em Journal of Applied Probability}, 8(2):344--356.

\bibitem[\protect\astroncite{May and Anderson}{1983}]{May1983}
May, R.~M. and Anderson, R.~M. (1983).
\newblock Epidemiology and genetics in the coevolution of parasites and hosts.
\newblock {\em Proc R Soc Lond B Biol Sci}, 219(1216):281--313.

\bibitem[\protect\astroncite{Rand}{1999}]{Rand1999}
Rand, D.~A. (1999).
\newblock Correlation equations and pair approximations for spatial ecologies.
\newblock {\em CWI Quarterly}, 12(3\&4):329--368.

\bibitem[\protect\astroncite{Ross}{2007}]{Ross2007}
Ross, S. (2007).
\newblock {\em Introduction to Probability Models}, chapter 6.4.
\newblock Academic Press, 9th edition.

\bibitem[\protect\astroncite{Simon and Kiss}{2013}]{simon2010}
Simon, P.~L. and Kiss, I.~Z. (2013).
\newblock From exact stochastic to mean-field {ODE} models: a case study of
  three different approaches to prove convergence results.
\newblock {\em IMA Journal of Applied Mathematics}, 78(5):945--964.

\bibitem[\protect\astroncite{Volz}{2008}]{Volz2008}
Volz, E. (2008).
\newblock {SIR} dynamics in random networks with heterogeneous connectivity.
\newblock {\em J Math Biol}, 56(3):293--310.

\end{thebibliography}

\appendix
\section*{Appendix}
\begin{lemma}[Lemma 6]
Consider the initial value problems $x'=f_1(t,x)$, $x(0)=x_1$ and $x'=f_2(t,x)$, $x(0)=x_2$ with solutions $\varphi_1(t)$ and $\varphi_2(t)$ respectively.  If $f_1$ is Lipschitz in $x$ with constant $L$ and $\norm{f_1(t,x)-f_2(t,x)}\le M$, then $\norm{\varphi_1(t)-\varphi_2(t)}\le (\norm{x_1-x_2}+M/L)e^{Lt}-M/L$.
\end{lemma}
\begin{proof}
Using the integral form of the differential equations,
\[ \norm{\varphi_1(t)-\varphi_2(t)}
 = \norm{ x_1-x_2+\int_0^t f_1(u,\varphi_1(u)) du - \int_0^t f_2(u,\varphi_2(u)) du}.\]
Using the triangle inequality,
\[ \norm{\varphi_1(t)-\varphi_2(t)}
 \leq \norm{x_1-x_2}+\norm{\int_0^t f_1(u,\varphi_1(u)) du - \int_0^t f_2(u,\varphi_2(u)) du}.\]
Using the integral form of the triangle inequality,
\[ \norm{\varphi_1(t)-\varphi_2(t)}
 \leq \norm{ x_1-x_2} + \int_0^t \norm{f_1(u,\varphi_1(u)) - f_2(u,\varphi_2(u))} du.\]
Applying the triangle inequality again,
\begin{multline*}
 \leq \norm{ x_1-x_2} \\
 +\int_0^t \bigl(\norm{f_1(u,\varphi_1(u))-f_1(u,\varphi_2(u))} + \norm{f_1(u,\varphi_2(u))-f_2(u,\varphi_2(u))}\bigr) du.
\end{multline*}
Applying the assumptions of this lemma,
\[ \norm{\varphi_1(t)-\varphi_2(t)} \leq \norm{ x_1-x_2} + \int_0^t (L \norm{\varphi_1(u)-\varphi_2(u)} + M) du.\]
Applying a specialized form of \nameref{Gronwall} then proves the claim.
\end{proof}

\begin{lemma}[Gronwall's Inequality]\label{Gronwall}
Suppose for $t\geq 0$, that $\theta(t)$ is a continuous nonnegative function; $L,M\geq 0$; and $\theta(t)\leq \theta(0)+\int_0^t (L\theta(u)+M) du$. Then for $t\geq 0$, $\theta$ is bounded by the solution to the initial value problem $x'=Lx+M$, $x(0)=\theta(0)$:
\[ \theta(t) \leq (\theta(0)+M/L)e^{Lt}-M/L.\]
\end{lemma}
\begin{proof}This result can be found in any graduate ODE text such as Lemma 6.2 in \citet{hale2009ordinary}.
\end{proof}

\end{document}